\documentclass[final]{siamltex}
\usepackage{graphicx}
\usepackage{amsmath}
\usepackage{amssymb}
\usepackage{url}

\newcommand{\er}{\epsilon}
\newcommand{\ex}{e}
\newcommand{\xx}{\mathbf{x}}
\newcommand{\yy}{\mathbf{y}}
\newcommand{\tr}{\ast}

\newcommand{\bb}{\mathbf{b}}

\newcommand{\zz}{\mathbf{z}}

\newcommand{\uu}{\mathbf{u}}
\newcommand{\vv}{\mathbf{v}}
\newcommand{\ww}{\mathbf{w}}
\newcommand{\rr}{\mathbf{r}}

\newcommand{\Tr}{{\text{\upshape Tr}}}

\newcommand{\KK}{\mathbb{C}}

\newcommand{\Span}{\mathrm{Span}}

\newcommand{\bigO}{\text{\upshape O}}

\newtheorem{algorithm}{Algorithm}
\newtheorem{remark}[theorem]{Remark}
\newtheorem{assumption}{Assumption}

\title{Computing Isolated Singular Solutions of Polynomial Systems:
Case of Breadth One\thanks{\quad This research is  supported
 by the Chinese National Natural Science Foundation under grant
    NSFC60821002/F02, 60911130369 and 10871194.} } 

\author{Nan Li         and
        Lihong Zhi  \thanks{ Key Lab of Mathematics Mechanization, AMSS, Beijing 100190,
        China,
        linan08@amss.ac.cn, lzhi@mmrc.iss.ac.cn}}

\begin{document}
\maketitle

\begin{abstract}
We present a symbolic-numeric method to refine an approximate
isolated singular solution $\hat{\mathbf{x}}=(\hat{x}_{1}, \ldots,
\hat{x}_{n})$ of a polynomial system $F=\{f_1, \ldots, f_n\}$ when
the Jacobian matrix of $F$ evaluated at $\hat{\mathbf{x}}$ has
corank one approximately. Our new approach is based on the
regularized Newton iteration and the computation of approximate Max
Noether conditions satisfied at the approximate singular solution.
The size of matrices involved in our algorithm is bounded by $n
\times n$. The algorithm converges quadratically if  $\hat{\xx}$ is
close   to the isolated exact  singular solution.
\end{abstract}


\begin{AMS}\end{AMS}

\pagestyle{myheadings}
\thispagestyle{plain}

\section{Introduction}
\label{sec:intro}
\paragraph{\bf{Motivation and problem statement}}
Consider an ideal $I$ generated by a polynomial system $F=\{f_1,
\ldots, f_n\}$, where $f_i\in \mathbb{C}[x_1,\ldots, x_n]$.
Suppose $\hat \xx=\hat{\xx}_{\ex}+\hat{\xx}_{\er}$, where
$\hat{\xx}_{\ex}$ denotes the isolated exact singular solution of
$F$ and $\hat{\xx}_{\er}$ denotes the error in the solution. The
{\it multiplicity} $\mu$ of $\hat{\xx}_{\ex}$ is defined as
$\mu=\dim(\mathbb{C}[\xx]/Q)$, where  $Q$ is the isolated primary
component whose associate prime ideal is $P=(x_1-\hat{x}_{1,\ex},
\ldots, x_n-\hat{x}_{n,\ex})$,  and  the {\it index} $\rho$ of
$\hat{\xx}_{\ex}$ is defined as the minimal nonnegative integer
$\rho$ such that $P^\rho \subseteq Q$ \cite{Waerden:1970}.

In \cite{WuZhi:2008, WuZhi:2009}, they compute the truncated
coefficient matrix of the involutive system to the order $\rho$, and
generate the multiplication matrices from its approximate null
vectors. Then
a basis of the approximate Max Noether space (Definition \ref{dfn1})
of $I$ at $\hat{\xx}$ can be obtained from these vectors (
 Theorem
5.4 in \cite{WuZhi:2008}).
Let $\hat \yy$ be the vector whose $i$-th element is the average of
the trace of the multiplication matrix with respect to $x_i$. In
\cite{WuZhi:2009}, it has been proved 
 that if the given approximation
$\hat{\xx}$ 
satisfies $ \|\hat \xx- \hat \xx_{\ex}\|=\varepsilon$, for a small
positive number $\varepsilon$, and the index $\rho$ and the
multiplicity $\mu$ are computed correctly, then the refined solution
obtained by adding $\hat \yy$ to $\hat \xx$ will satisfy $\|\hat
\xx+\hat \yy- \hat \xx_{\ex}\|= \bigO(\varepsilon^2) $. Here and
hereafter, $\|\cdot\|$ is denoted as the $l^2$-norm. The size of
these coefficient matrices in \cite{WuZhi:2009} is bounded by
$n\tbinom{\rho+n}{n}\times \tbinom{\rho+n}{n}$ which will be very
big when $\rho$ is large. Especially, when the corank of the
Jacobian
$F'(\hat{\xx}_{\ex})$ is one,
then
$\rho=\mu$, which is also called the breadth one case in~\cite{
DLZ:2009,DaytonZeng:2005}.

In \cite{LiZhi:2009}, we present a new algorithm which is based on
Stetter's strategies \cite{Stetter:2004} for computing a closed
basis $L=\{L_0, \ldots, L_{\mu-1}\}$ of the approximate Max Noether
space of $I=(f_1, \ldots, f_n)$ at $\hat \xx$ incrementally in the
breadth one case. The size of matrices we used in computing each
order of Max Noether conditions is bounded by $n \times (n-1)$,
which does not depend on the multiplicity. Moreover, during the
computation, we only need to store the input polynomial system $F$,
 the last $n-1$ columns of the Jacobian $F'(\hat{\xx})$ and
the computed Max Noether conditions. Therefore, in the breadth one
case, both storage space and execution time for computing  a closed
basis of the approximate Max Noether space are reduced significantly by the
algorithm in \cite{LiZhi:2009}.   This motivates us to consider
whether we can get rid of  large coefficient matrices in
\cite{WuZhi:2008, WuZhi:2009} and refine approximate singular
solutions more efficiently based on the computed Max Noether
conditions.

\paragraph{\bf{Main contribution}}

Suppose we are given an approximate  singular solution $\hat{\xx}$
of a polynomial system $F$ satisfying $ \|\hat \xx- \hat
\xx_{\ex}\|=\varepsilon$,  where the positive number $\varepsilon$
is small enough such that  there are no other solutions of $F$
nearby. We also assume that the corank of the Jacobian matrix
$F'(\hat{\xx}_{\ex})$ is one.  In order to restore the quadratic
convergence of the Newton method, we first apply one regularized
Newton iteration (in Section \ref{sec:31}) to obtain a new
approximation $\hat \xx+\hat \yy$ which also satisfies the
assumptions above, and then compute the approximate null vector
$\rr_1$ of the Jacobian $F'(\hat \xx+\hat \yy)$ which gives
 a generalized Newton direction, and  the step length $\delta$ is
obtained  by solving a linear system formulated by the computed
differential operators using the algorithm in \cite{LiZhi:2009}. We
show that $\|\hat \xx+\hat \yy+ \delta \rr_1- \hat \xx_{\ex}
\|=\bigO(\varepsilon^2)$. The size of matrices involved in our
algorithm is bounded by $n\times n$.
The method has been
implemented in Maple. Moreover, we also prove the conjecture in
\cite{DaytonZeng:2005} that  the breadth one depth-deflation always
terminates at step $\mu-1$, where $\mu$ is the multiplicity.

\paragraph{\bf{Structure of the paper}}
Section \ref{sec:prelim} is devoted to recall some notations and
well-known facts.
In Section \ref{sec:refineroot}, we describe an 
algorithm for refining approximate isolated singular solutions of polynomial systems in the breadth one case.
Moreover, we prove that the algorithm 
 converges quadratically if the
approximate solution is 
close to the isolated exact  
singular solution. 
Some experiment results are given in Section \ref{lsec:complexity}.
We mention some ongoing research  in Section \ref{lsec:conclusion}.

\section{Preliminaries}\label{sec:prelim}
Let $D(\alpha)=D(\alpha_1,
\ldots, \alpha_n): \mathbb{C}[\xx]\rightarrow \mathbb{C}[\xx]$
denote the differential operator defined by:
\begin{equation*}
D(\alpha_1,\ldots,\alpha_n):=\frac{1} {\alpha_1!\cdots
\alpha_n!}\frac{\partial^{\alpha_1+\cdots+\alpha_n}}{\partial
x_1^{\alpha_1}\cdots
\partial x_n^{\alpha_n}},
\end{equation*}
for nonnegative integer array $\alpha=[\alpha_1, \ldots,
\alpha_n]$.
 We write
$\mathfrak{D}=\{D(\mathbf{\alpha}), ~|\mathbf{\alpha}| \geq 0 \}$
and denote by $\Span_\mathbb{C}(\mathfrak{D})$ the
$\mathbb{C}$-vector space generated by $\mathfrak{D}$. Introducing
a morphism on $\mathfrak{D}$ that acts as ``integral":
\begin{align*}
\Phi_j(D(\alpha)):=\left\{
\begin{array}{ll}
D(\alpha_1, \ldots, \alpha_{j}-1, \ldots, \alpha_n), & \mbox{if $\alpha_{j} > 0$,}\\
0, & \mbox{otherwise.}
\end{array}
\right.
\end{align*}

As a counterpart of the anti-differentiation operator $\Phi_{j}$, we
define the differential operator $\Psi_{j}$  as
\begin{align*}
\Psi_{j}(D(\alpha)):=D(\alpha_1, \ldots, \alpha_{j}+1, \ldots,
\alpha_n).
\end{align*}

\begin{definition}\label{dfn1}
Given a zero $\hat{\xx}_{\ex}$ of an ideal $I=(f_1, \ldots, f_n)$,
we define the Max Noether space~\cite{MollerTenberg:2001} of $I$ at
$\hat{\xx}_{\ex} $ as
\begin{equation*}
\triangle_{\hat{\xx}_{\ex} }(I):=\{L\in
\Span_\mathbb{C}(\mathfrak{D})|\, L(f) _{\xx=\hat{\xx}_{\ex} }=0,
~\forall f\in I\}.
\end{equation*}
\end{definition}
Conditions equivalent to $L(f) _{\xx=\hat{\xx}_{\ex} }=0, ~\forall L
\in \triangle_{\hat{\xx}_{\ex} }(I)$ are called {\it Max Noether
conditions} \cite{MollerTenberg:2001}. The space $\{
L_{\hat{\xx}_{\ex}} ~|~L \in \triangle_{\hat{\xx}_{\ex}}(I) \}$,
where $L_{\hat{\xx}_{\ex}}(f):=L(f)_{\xx=\hat{\xx}_{\ex}}$, is also
called the {\it dual space} of $I$ at $\hat {\xx}_{\ex}$
\cite{DaytonZeng:2005,MMM:1993, MMM:1996, Stetter:1995,
Mourrain:1996, Stetter:2004}.  For a nonnegative integer $k$,
$\triangle_{\hat{\xx}_{\ex} }^{(k)}(I)$ is a subspace of
$\triangle_{\hat{\xx}_{\ex} }(I)$ which consists of differential
operators with differential order bounded by $k$. Obviously,
$\triangle_{\hat{\xx}_{\ex} }^{(0)}(I)=D(0,\ldots,0)$. We have that
\begin{equation}
\dim_\KK(\triangle_{\hat{\xx}_{\ex} }(I))=\mu,
\end{equation}
 where $\mu$ is the
multiplicity of the zero $\hat{\xx}_{\ex} $.

\begin{definition}\cite{MMM:1993,Stetter:2004}
A subspace $\triangle$ of $\Span_\mathbb{C}(\mathfrak{D})$ is said
to be closed if and only if its dimension is finite and
\begin{equation*}
L \in \triangle  \Longrightarrow \Phi_{j}(L) \in
\triangle, ~j=1,\ldots, n.\end{equation*}
\end{definition}
Suppose $\Span(L_0, L_1, \ldots, L_{\mu-1})$ is closed and $L_0,
\ldots, L_{\mu-1}$ are linearly independent differential operators
which satisfy that $L_i(f_j) _{\xx=\hat{\xx}_{\ex} }=0, ~j=1, \ldots, n,
~i=0, \ldots,{\mu-1}$, then due to the closedness, $L_i(q \cdot f_j)
_{\xx=\hat{\xx}_{\ex} }=0, ~\forall q
 \in \mathbb{C}[x_1,\ldots, x_n]$. Hence,
 $\triangle_{\hat{\xx}_{\ex} }(I)=\Span( L_0, L_1, \ldots,
 L_{\mu-1})$.

\begin{lemma}\label{lemma:corankone}
Let $F'(\hat{\xx}_{\ex} )$ be the Jacobian of a polynomial
system $F=\{f_1, \ldots, f_n\}$ evaluated at $\hat{\xx}_{\ex} $.
Suppose the corank of $F'(\hat{\xx}_{\ex} )$ is one, i.e., the
dimension of its null space is one, then
$\dim(\triangle_{\hat{\xx}_{\ex}
}^{(k)}(I))=\dim(\triangle_{\hat{\xx}_{\ex} }^{(k-1)}(I))+1 $ for $1
\leq k \leq \mu-1$ and $\dim(\triangle_{\hat{\xx}_{\ex}
}^{(k)}(I))=\dim(\triangle_{\hat{\xx}_{\ex} }^{(\mu-1)}(I))$,
 for $ k \geq \mu$. Hence we have $\mu=\rho$.
\end{lemma}

\begin{proof}
 Lemma~\ref{lemma:corankone}  is  an immediate consequence of
\cite[Theorem 2.2]{Stanley:1973} and \cite[Lemma
1]{DaytonZeng:2005}.
\end{proof}

\begin{theorem}\label{LiZhi}\cite{LiZhi:2009}
Suppose we are given an isolated multiple root $\hat{\xx}_{\ex} $ of
the polynomial system $F=\{f_1, \ldots, f_n\}$ with the multiplicity
$\mu$ and the corank of the Jacobian $F'(\hat{\xx}_{\ex} )$
is one, and $L_1=D(1,0, \ldots, 0) \in \triangle_{\hat{\xx}_{\ex}
}^{(1)}(I)$. We can construct the $k$-th order Max Noether condition
retaining the closedness incrementally for $k$ from $2$ to $\mu-1$
by the following formulas:
\begin{equation}\label{kthcondition}
L_k= P_k+ a_{k,2} D(0,1,\ldots,0)+\cdots+a_{k,n}D(0,\ldots,1),
\end{equation}
where $P_k$ has no free parameters and is obtained from previous
computed $L_1, \ldots, L_{k-1}$ by the following formula:
\begin{equation}\label{fixedkthcondition}
P_k=\Psi_1(L_{k-1})+\Psi_2(Q_{k,2})_{\alpha_1=0}+\cdots+\Psi_n(Q_{k,n})_{\alpha_1=\alpha_2=\cdots=\alpha_{n-1}=0},
\end{equation}
where
\begin{equation}\label{phikth}
\Phi_1(P_k)=L_{k-1}, ~Q_{k,j}=\Phi_{j}(P_k)=a_{2,j} L_{k-2}+\cdots+a_{k-1,j}
 L_1, ~2 \leq j \leq n.
\end{equation}
Here $\Psi_{j}(Q_{k,j})_{\alpha_1=\cdots=\alpha_{j-1}=0}$ means that we
only pick up differential operators $D(\alpha)$ in $Q_{k,j}$ where $\alpha_1=\cdots=\alpha_{j-1}=0$. The
parameters $a_{k,j},~j=2,\ldots, n$  are determined by checking
whether $[P_k(f_1) _{\xx=\hat{\xx}_{\ex} }, \ldots, P_k(f_n)
_{\xx=\hat{\xx}_{\ex} }]^T$ can be written as a linear combination of  the
last $n-1$ linearly independent columns of
$F'(\hat{\xx}_{\ex} )$.
\end{theorem}

Suppose $\hat \xx$ is an approximation of $\hat{\xx}_{\ex}$ and
$\|\hat{\xx}-\hat{\xx}_{\ex}\|=\varepsilon \ll 1$, we can use
the algorithm {\sf MultiplicityStructureBreadthOneNumeric} in
\cite{LiZhi:2009} to compute a closed basis $\{L_0,
\ldots, L_{\mu-1}\}$ of the approximate Max Noether space of $I$ at $\hat{\xx}$.
Since the errors in the matrix of the linear system
\[
\left[P_k(F)_{\xx=\hat{\xx}},\frac{\partial F(\hat{\xx})}{\partial x_2},\ldots,\frac{\partial F(\hat{\xx})}{\partial x_n}\right]\cdot [1,a_{k,2},\ldots,a_{k,n}]^T=0,
\]
used in Theorem \ref{LiZhi}
are bounded by  $\bigO(\varepsilon)$ and
\[
L_k= P_k+ a_{k,2}
D(0,1,0,\ldots,0)+\cdots+a_{k,n}D(0,\ldots,0,1),\]
is determined by its right singular vector $[1,a_{k,1},\ldots,a_{k,n}]^T$ corresponding to
its smallest singular value, we have
\begin{equation}\label{aL_k}
\|L_k(F)_{\xx=\hat{\xx}}\|=\bigO(\varepsilon),
\end{equation}
 according to
\cite[Corollary 8.6.2]{Golub:1996}.
\section{An Algorithm for Refining Approximate Singular Solutions}
\label{sec:refineroot}Suppose  we are given
 an approximate solution
\begin{equation*}
 \hat
\xx=\hat{\xx}_{\ex}+\hat{\xx}_{\er}, \end{equation*}
 where  $\hat{\xx}_{\er}$ denotes the error in the solution and
 $\hat{\xx}_{\ex}$
denotes the exact  solution of the polynomial system $F=\{f_1,
\ldots, f_n\}$ with the multiplicity $\mu$  and the index $\rho$. In
this section, we present a new method to refine $\hat \xx$  in the
breadth one case, i.e., $\mu=\rho$.

Let $A=F'(\hat{\xx})$ be the  Jacobian matrix of $F$ evaluated at
$\hat \xx$ and $\bb=-F(\hat{\xx})$.
Suppose the error in the solution is small enough, i.e., $ \|\hat{\xx}-\hat{\xx}_{\ex}\|=\varepsilon\ll 1$, 
and $A$ is invertible, then Newton's iteration computes
\begin{equation}\label{newton}
\hat \yy= {A}^{-1} \bb,
\end{equation}
and $\|\hat \xx+\hat \yy-\hat \xx_{\ex} \|=\bigO(\varepsilon^2)$
according to the well-known Kantorovich theorem \cite{KA64}.
However, if ${A}$ is singular, then the convergence of  Newton
iterations is linear rather than quadratic.

Rall~\cite{Rall66}  studied  the convergence properties of Newton's
method at singular points. Some modifications of Newton's method to
restore quadratic convergence have also been proposed
in~\cite{Chen97,DeckerKelley:1980I,DeckerKelley:1980II,DeckerKelley:1982,
Griewank:1980, Griewank85, GriewankOsborne:1981, Ojika:1987,
OWM:1983, Reddien:1978,Reddien:1980,  ShenYpma05}.
In~\cite{Griewank85},  a bordered system was introduced to restore
the quadratic convergence of Newton's method when ${A}$ has corank
one approximately and $\hat \xx$ is a simple singular solution.  It
is clear to see  that the regularity condition in~\cite{Griewank85}
can not be satisfied if the multiplicity is larger than $2$.

For simplicity, we make an assumption throughout this  section.


\begin{assumption}\label{assump}
 Suppose we
are given an approximate  singular solution $\hat{\xx}$ of a
polynomial system $F$ satisfying $ \|\hat \xx- \hat
\xx_{\ex}\|=\varepsilon$,  where the positive number $\varepsilon$
is small enough such that  there are no other solutions of $F$
nearby. Moreover, we assume that the corank of the Jacobian matrix
$F'(\hat{\xx}_{\ex})$ is one.
\end{assumption}

 Let
$A=F'(\hat{\xx})$ be the  Jacobian matrix of $F$ evaluated at $\hat
\xx$ and its singular values be $\sigma_1, \ldots, \sigma_s$. Under
Assumption \ref{assump}, we have
$\|F(\hat{\xx})\|=\bigO(\varepsilon)$, $\sigma_{i}=\Theta(1), ~1
\leq i \leq n-1$ and $\sigma_n=\bigO(\varepsilon)$.

\begin{remark}
Notice that the  notation $\bigO(g)$  denotes that the value is
bounded above by $g$ up to a constant factor, while $\Theta(g)$
denotes that the value is bounded both above and below by $g$ up to
constant factors.
\end{remark}

\subsection{Regularized Newton iteration}
\label{sec:31}

Under Assumption \ref{assump},   $F'(\hat \xx)$ is approximately
singular. Instead of using (\ref{newton}) to compute $\hat \yy$, we
apply
 Tikhonov regularization~\cite{Tychonoff77} to solve the minimization problem
\begin{equation*}
\mathrm{min}~\|{A}\yy-\bb\|^2+ \lambda \|\yy\|^2,
\end{equation*}
to obtain $\hat \yy$, where $A=F'(\hat \xx)$ and $\bb=-F(\hat{\xx})$. The real number
$\lambda>0$ is called the regularization parameter.

\begin{theorem}[{\sf Regularized Newton Iteration}]\label{thm:prenewton}
Under Assumption~\ref{assump}, if we choose the smallest singular
value $\sigma_n$ of $F'(\hat \xx)$ as the regularization parameter,
the solution $\hat{\yy}$ of the following regularized least squares
problem
\begin{equation}\label{Lagrange}
({A}^{\tr} {A}+\sigma_n I_n)\hat{\yy} ={A}^{\tr} \bb
\end{equation}
satisfies
\begin{equation}\label{updatey}
\|\hat \yy\|=\bigO(\varepsilon), ~
\|F(\hat{\xx}+\hat{\yy})\|=\bigO(\varepsilon^2),
\end{equation}
where $A^{\tr}$ is the Hermitian (conjugate) transpose of $A=F'(\hat
\xx)$, $I_n$ is the $n\times n$ identity matrix and $\bb=-F(\hat{\xx})$.
\end{theorem}

\begin{proof}
Suppose ${A}=U \cdot \Sigma \cdot V^{\tr}$ is the singular value
decomposition of ${A}$
 where $\Sigma= \diag\{\sigma_1,\ldots,\sigma_n\}$, then the
 solution of (\ref{Lagrange}) is
\begin{equation}\label{yy}
 \hat \yy=V \cdot ({\Sigma}^2+\sigma_n I_n)^{-1}\cdot\Sigma\cdot
 U^{\tr}\cdot\bb.
 \end{equation}
Since $\sigma_{i}=\Theta(1), ~1 \leq i \leq n-1$,
$\sigma_n=\bigO(\varepsilon)$ and $\|\bb\|=\bigO(\varepsilon)$, we
have
\begin{equation*}
\|\hat \yy\|^2=\sum_{i=1}^{n} \left(\frac{\sigma_i
|\tilde{b}_i|}{\sigma_{i}^{2}+\sigma_n}\right)^2=\bigO(\varepsilon^2),
\end{equation*}
where  $\tilde{\bb}=[\tilde b_1, \ldots, \tilde b_n]^T=U^{\tr} \bb$
and $\|\tilde{\bb}\|=\|\bb\|=\bigO(\varepsilon)$. Hence, $\|\hat
\yy\|=\bigO(\varepsilon)$.

From the Taylor expansion of $F$ at $\hat \xx$, we have \[ F(\hat
\xx_{\ex})=-\bb+A(\hat \xx_{\ex}-\hat{\xx})+\bigO(\varepsilon^2).\]
Hence
\[\|-\bb+A(\hat \xx_{\ex}-\hat{\xx})\|=\bigO(\varepsilon^2).\]
Furthermore, we have
\[\|-U^{\tr} \bb+\Sigma\cdot V^{\tr}(\hat \xx_{\ex}-\hat{\xx})\|=\bigO(\varepsilon^2).\]
Since $\sigma_n=\bigO(\varepsilon)$ and
$\|V^{\tr}(\hat{\xx}-\hat{\xx}_{\ex})\|=\|\hat{\xx}-\hat{\xx}_{\ex}\|=\varepsilon$, we derive that the last
component of the vector $\tilde{\bb}$ satisfies
\begin{equation}\label{bs}
|\tilde{b}_n|=\bigO(\varepsilon^2).
\end{equation}
Since
\[A \hat
\yy-\bb=U\cdot\diag\left\{\frac{-\sigma_n}{\sigma_{1}^{2}+\sigma_n},
\ldots, \frac{-\sigma_n}{\sigma_{n}^{2}+\sigma_n}\right\}\cdot
\tilde{\bb},
\]
we have
\[\|A \hat
\yy-\bb\|^2=\sum_{i=1}^{n}\left(\frac{\sigma_n
|\tilde{b}_i|}{\sigma_{i}^{2}+\sigma_n}\right)^2,\]
where
\[\frac{\sigma_n}{\sigma_{i}^{2}+\sigma_n}=\bigO(\varepsilon), ~
{\rm for} ~ i=1,\ldots, n-1,\]
 and $\|\tilde{\bb}\|=\bigO(\varepsilon)$.
Although \[\frac{\sigma_n}{\sigma_{n}^{2}+\sigma_n}=\Theta(1),\] we
have from (\ref{bs}) that $|\tilde{b}_{n}|=\bigO(\varepsilon^2)$, hence
\begin{equation}\label{res}
\|A \hat \yy-\bb\|=\bigO(\varepsilon^2).
\end{equation}
Finally, from the Taylor expansion of $F$ at $\hat \xx$, we have
\[
\|F(\hat{\xx}+\hat{\yy})\|\leq
\|-\bb+{A}\hat{\yy}\|+\bigO(\varepsilon^2)=\bigO(\varepsilon^2).
\]
\end{proof}

 According to Theorem~\ref{thm:prenewton}, after  applying one regularized Newton
iteration to $F$ and $\hat{\xx}$, we  get $\hat{\yy}$
satisfies (\ref{updatey}), and  the new approximate singular
solution $\hat{\xx}+\hat{\yy}$ satisfies
\begin{equation*}
\|\hat{\xx}+\hat{\yy}-\hat{\xx}_{\ex}\|\leq
\|\hat{\xx}-\hat{\xx}_{\ex}\|+\|\hat{\yy}\|=\varepsilon+\bigO(\varepsilon).
\end{equation*}
If \begin{equation*}
\|\hat{\xx}+\hat{\yy}-\hat{\xx}_{\ex}\|=\bigO(\varepsilon^2),
\end{equation*}
 then
we have already achieved the quadratic convergence. However,  the
convergence rate of the regularized Newton iteration is linear
too when  the Jacobian matrix is near singular.  Hence, in most
cases, we will have
\begin{equation}\label{rnewton}
 \|\hat{\xx}+\hat{\yy}-\hat{\xx}_{\ex}\|=\Theta(\varepsilon).
\end{equation}
We  show below how to restore  the quadratic convergence when the
computed approximate singular solution $\hat \xx+ \hat \yy$
satisfies (\ref{updatey}) and  (\ref{rnewton}).

 If $L_1 \in \triangle_{\hat{\xx}+\hat \yy }^{(1)}(I)$ is not $D(1,0, \ldots,
0)$, as pointed out by Stetter in \cite{Stetter:2004},
 we  can compute the right singular vector of $F'(\hat{\xx}+\hat \yy )$ corresponding
to its smallest singular value $\sigma'_n$, denoted by $\rr_1$ satisfying  $\|\rr_1\|=1$ and
\begin{equation}\label{definer1}
\|F'(\hat{\xx}+\hat \yy ) \, \rr_1\|=\sigma'_n=
\bigO(\varepsilon).\end{equation}

Let us form a unitary  matrix $R=[\rr_1, \ldots, \rr_n]$ and perform
the linear transformation
\begin{equation}\label{zzz1}
H(\zz)=F(R\, \zz).
\end{equation}
It is clear that
\begin{equation}\label{zzz2}
\hat{\zz}_{\ex}=R^{-1}\hat{\xx}_{\ex}
\end{equation}
is an exact root of $H(\zz)$  and
\begin{equation}\label{zzz22}
\quad\hat{\zz}=R^{-1} \, ({\hat{\xx}+ \hat \yy })
\end{equation}
is an approximate root of $H(\zz)$. Moreover,
we have
\begin{equation}\label{zzzold4}
\|\hat{\zz}-\hat{\zz}_{\ex}\|=
\|R^{-1}(\hat{\xx}+\hat{\yy}-\hat{\xx}_{\ex})\|=\|\hat{\xx}+\hat{\yy}-\hat{\xx}_{\ex}\|=\Theta(\varepsilon),
\end{equation}
\begin{equation}\label{zzz4h}
\|H(\hat \zz)\|=\|F(\hat \xx+\hat \yy)\|=\bigO(\varepsilon^2),
\end{equation}
and
\begin{equation}\label{zzz5}
\left\|\frac{\partial H(\hat \zz)}{\partial z_1}\right\| =\left\|
F'(\hat{\xx}+\hat \yy ) \rr_1  \right\|=\sigma'_n =\bigO(\varepsilon).
\end{equation}
Hence, the condition (\ref{rnewton}) is equivalent to
(\ref{zzzold4}). Here and hereafter, we always assume that  $\hat
\zz$ satisfies
\begin{equation}\label{zzz4}
\|\hat{\zz}-\hat{\zz}_{\ex}\|=\Theta(\varepsilon).
\end{equation}

\begin{theorem}\label{lineartrans}
The root $\hat{\zz}_{\ex}$ defined in (\ref{zzz2}) is an isolated
singular solution of $H$ with the  multiplicity $\mu$ and the corank
of $H'(\hat{\zz}_{\ex})$ is one.
\end{theorem}
\begin{proof}
Since $H'(\hat{\zz}_{\ex})=F'(\hat{\xx}_{\ex})R$ and $R$ is a
unitary matrix, we derive that the corank of $H'(\hat{\zz}_{\ex})$
is one. Let $\mu'$ be the multiplicity of $\hat{\zz}_{\ex}$, and
$\{L_0,L_1,\ldots,L_{\mu'-1}\}$ be a closed basis of the Max Noether
space of $H$ at $\hat{\zz}_{\ex}$. The operator $\Gamma_R:
\Span_\mathbb{C}(\mathfrak{D})\rightarrow\Span_\mathbb{C}(\mathfrak{D})$
is defined by:
\begin{eqnarray*}
\Gamma_R(D(\alpha)):&&=\Gamma_R\left(\frac{1} {\alpha_1!\cdots
\alpha_n!}\frac{\partial^{\alpha_1+\cdots+\alpha_n}}{\partial
z_1^{\alpha_1}\cdots
\partial z_n^{\alpha_n}}\right)\\
&&=\frac{1} {\alpha_1!\cdots
\alpha_n!}\frac{\partial^{\alpha_1+\cdots+\alpha_n}}{\partial
{(\rr_{1}^{\tr}\cdot \xx)}^{\alpha_1}\cdots
\partial {(\rr_{n}^{\tr}\cdot \xx)}^{\alpha_n}}\\
&&=\frac{1} {\alpha_1!\cdots
\alpha_n!}\sum_{|\beta|=|\alpha|}c_{\beta}\cdot
\frac{\partial^{\beta_1+\cdots+\beta_n}}{\partial
x_1^{\beta_1}\cdots
\partial x_n^{\beta_n}},\\
&&=\frac{1} {\alpha_1!\cdots
\alpha_n!}\sum_{|\beta|=|\alpha|}c_{\beta}\cdot \beta_1!\cdots
\beta_n!\cdot D(\beta),
\end{eqnarray*}
where $c_{\beta}$ is the coefficient of $\frac{\partial^{\beta_1+\cdots+\beta_n}}{\partial
x_1^{\beta_1}\cdots
\partial x_n^{\beta_n}}$ in the expansion of $\frac{\partial^{\alpha_1+\cdots+\alpha_n}}{\partial
{(\rr_{1}^{\tr}\cdot \xx)}^{\alpha_1}\cdots
\partial {(\rr_{n}^{\tr}\cdot \xx)}^{\alpha_n}}$.
Since $H(\zz)=F(R\,\zz)$ and $\xx=R\,\zz$, according to multivariate chain rules, we have
\[\Gamma_R(L_k)(F)_{\xx=\hat{\xx}_{\ex}}=L_k(H)_{\zz=\hat{\zz}_{\ex}}=0,\]
and for $1\leq j\leq n$,
\begin{align*}
\Phi_j(\Gamma_R(L_k))=&\Gamma_R\left(\sum_{i=1}^{n}r_{i,j}\Phi_i(L_k)\right)\\
=&\Gamma_R\left(\sum_{i=1}^{k-2}(a_{k-i,2}r_{2,j}+\cdots+a_{k-i,n}r_{n,j})L_i\,+r_{1,j}L_{k-1}\right)\\
=&\sum_{i=1}^{k-1}(a_{k-i,2}r_{2,j}+\cdots+a_{k-i,n}r_{n,j})\Gamma_R(L_i)\,+r_{1,j}\Gamma_R(L_{k-1}),
\end{align*}
where $0\leq k\leq \mu'-1$. Hence,
$\{\Gamma_R(L_0),\Gamma_R(L_1),\ldots,\Gamma_R(L_{\mu'-1})\}$ is a
closed basis of $\triangle_{\hat{\xx}_{\ex}}^{(\mu'-1)}(I)$ and
$\mu'\leq\mu$. On the other hand, since $F(\xx)=H(R^{-1}\,\xx)$, we
derive that $\mu\leq\mu'$. Hence, $\mu'=\mu$.
\end{proof}

\begin{remark}\label{remarklt}
Since $H'(\hat{\zz})=F'(\hat{\xx}+\hat{\yy})R$ and $R$ is a unitary
matrix, we derive that the singular values of $H'(\hat{\zz})$ are
the same as those of $F'(\hat{\xx}+\hat{\yy})$ and the corank of
$H'(\hat{\zz})$ is one approximately. Suppose
$\{L_0,L_1,\ldots,L_{\mu-1}\}$ is a closed basis of the
approximate Max Noether space of $H$ at $\hat{\zz}$, where $L_0=D(0,\ldots,0)$
and $L_1=D(1,0,\ldots,0)$. From the proof of Theorem
\ref{lineartrans} and (\ref{aL_k}), we have
\[\Gamma_R(L_k)(F)_{\xx=\hat{\xx}+\hat{\yy}}=L_k(H)_{\zz=\hat{\zz}}=\bigO(\varepsilon),\]
and
\[\Phi_j(\Gamma_R(L_k))\in\Span\{\Gamma_R(L_1),\ldots,\Gamma_R(L_{k-1})\},\]
where $0\leq k\leq \mu-1$ and $1\leq j\leq n$. Hence,
$\{\Gamma_R(L_0),\Gamma_R(L_1),\ldots,\Gamma_R(L_{\mu-1})\}$ is a
closed basis of $\triangle_{\hat{\xx}+\hat{\yy}}(I)$.
\end{remark}

\begin{remark}\label{regular}
It should be  noticed that  Theorem \ref{lineartrans}  holds as long
as $R$ is a regular matrix. However, if we choose a unitary matrix
$R$, then it is much easier to compute the inverse of $R$ since
$R^{-1}=R^{\tr}$.
\end{remark}

It is interesting to notice that, after running one
regularized Newton iteration, the last $n-1$ elements of the
solution $\hat \zz$ have already been
 refined quadratically.

\begin{theorem}{\label{zi}}
 Suppose  $\hat \zz_{\ex}$ and $\hat \zz$  are defined in (\ref{zzz2}) and (\ref{zzz22}) respectively.    Under Assumption~\ref{assump}, we have
\begin{equation}\label{newzzz3}
|\hat{z}_{1,\ex}-\hat{z}_1| =\Theta(\varepsilon),
\end{equation}
and
\begin{equation}\label{zzz3}
|\hat{z}_{i,\ex}-\hat{z}_i| =\bigO(\varepsilon^2), ~{\rm
for}~i=2,\ldots,n.
\end{equation}

\end{theorem}

\begin{proof}
From the Taylor expansion of $H(\zz)$ at $\hat \zz$, we have
\begin{equation*}
H(\hat{\zz}_{\ex})=H(\hat{\zz})+H'(\hat{\zz})(\hat{\zz}_{\ex}-\hat{\zz})
+\bigO(\varepsilon^2).
\end{equation*}
Since  $H(\hat \zz_{\ex})=0$ and $\|H(\hat
\zz)\|=\bigO(\varepsilon^2)$, we have
\[\|H'(\hat{\zz})(\hat{\zz}_{\ex}-\hat{\zz})\|=\bigO(\varepsilon^2).\]
From (\ref{zzz5}) and (\ref{zzz4}), we have
\[\left\| \frac{\partial H(\hat \zz)}{\partial z_1} \,
(\hat{z}_{1,\ex}-\hat{z}_1)\right \|=\bigO(\varepsilon^2),\]
and
\[\left\|\left[\frac{\partial H(\hat \zz)}{\partial z_2},\ldots,\frac{\partial H(\hat \zz)}{\partial z_n}\right]
\cdot[\hat{z}_{2,\ex}-\hat{z}_2,\ldots,\hat{z}_{n,\ex}-\hat{z}_n]^T\right
\|=\bigO(\varepsilon^2).\] According to Remark \ref{remarklt}, the
matrix $\left[\frac{\partial H(\hat \zz)}{\partial z_2},
 \ldots, \frac{\partial H(\hat \zz)}{\partial z_n}\right]$ is of full column rank,
so that
 (\ref{zzz3}) is correct.
 The
 equation (\ref{newzzz3}) follows from
 (\ref{zzz4}) and (\ref{zzz3}).
\end{proof}

If the multiplicity $\mu$ is larger than $2$, the regularity
assumption in~\cite{Griewank85} will not be satisfied. The violation
of the regularity assumption  is caused by the existence of  the
higher order Max Noether condition. It is interesting to notice that
the left singular vector of the Jacobian matrix $H'(\hat \zz)$
corresponding to the smallest singular value can be used to prove
the following theorem.

\begin{theorem}\label{J1}
If the multiplicity  of the singular root is larger than $2$, under
Assumption~\ref{assump}, we have
\begin{equation*}
\|L_1(H)_{\zz=\hat{\zz}}\|= \left\|\frac{\partial H(\hat \zz)}{
\partial z_1}\right\|=\bigO(\varepsilon^2).
\end{equation*}
\end{theorem}

\begin{proof}
If $\mu >2$, according to Theorem \ref{LiZhi} and (\ref{aL_k}),
 there exists
a second order Max Noether condition such that
\begin{equation}\label{L2}
\|L_2(H)_{\zz=\hat{\zz}}\|= \left\|\left(\frac{1}{2}\frac{\partial^2 }{
\partial z_1^2}+a_{2,2} \frac{\partial }{
\partial z_2}+ \cdots+a_{2,n} \frac{\partial }{
\partial z_n} \right)(H)_{\zz=\hat{\zz}}\right\|=\bigO(\varepsilon).
\end{equation}

Suppose $\uu_n$ is the left singular vector of $H'(\hat{\zz})$
corresponding to the smallest singular value $\sigma'_n$ and $\|\uu_n\|=1$,  then
\begin{equation}\label{uu1}
 \left|\uu_n^{\tr} \,\frac{\partial H(\hat{\zz})}{
\partial z_i}\right|=
\bigO(\varepsilon),  ~ 1 \leq i \leq n.
\end{equation}
From (\ref{L2}) and (\ref{uu1}), we have
\begin{equation}\label{L2uu}
 \left|\uu_n^{\tr} \, \frac{\partial^2 H(\hat{\zz})}{ \partial z_1^2}\right|
= \bigO(\varepsilon).
\end{equation}
 Therefore,  we get
\begin{equation}\label{L2uu1}
 \left|\uu_n^{\tr}  \frac{\partial^2 H(\hat{\zz})}{ \partial z_1^2}
 (\hat{z}_{1,\ex}-\hat{z}_1)^2\right|
= \bigO(\varepsilon^3).
\end{equation}

From the Taylor expansion of $H(\zz)$ at $\hat{\zz}$, we have
\begin{equation}\label{Taylor}
H(\hat{\zz}_{\ex})=H(\hat{\zz})+H'(\hat{\zz})(\hat{\zz}_{\ex}-\hat{\zz})
+H''(\hat{\zz})(\hat{\zz}_{\ex}-\hat{\zz})^2 +\bigO(\varepsilon^3),
\end{equation}
where $(\hat{\zz}_{\ex}-\hat{\zz})^2$ denotes the vector of all
monomials with degree $2$ and  $H''(\hat{\zz})$ consists of  all
second order derivatives of $H$ evaluated at $\hat \zz$.

According to  Theorem \ref{zi},   all elements in
$(\hat{\zz}_{\ex}-\hat{\zz})^2$ are $\bigO(\varepsilon^3)$ except
the first one. Combining with (\ref{L2uu1}),  we have
\begin{equation}\label{uu3}
\left|\uu_n^{\tr}\,
H''(\hat{\zz})(\hat{\zz}_{\ex}-\hat{\zz})^2\right|=\bigO(\varepsilon^3).
\end{equation}
On the other hand, the Taylor expansion of $H(\zz)$ at
$\hat{\zz}_{\ex}$ shows that
\begin{equation*}
H(\hat{\zz})=H(\hat{\zz}_{\ex})+H'(\hat{\zz}_{\ex})(\hat{\zz}-\hat{\zz}_{\ex})
+H''(\hat{\zz}_{\ex})(\hat{\zz}-\hat{\zz}_{\ex})^2+\bigO(\varepsilon^3).
\end{equation*}
Since the corank of $H'(\hat \zz_{\ex})$ is one,  suppose $\uu_{\ex}$ is
the left null  vector of $H'(\hat{\zz}_{\ex})$ and $\|\uu_{\ex}\|=1$, then
\[ \uu_{\ex}^{\tr} H'(\hat \zz_{\ex})=0.\]
Notice that
\[\|\uu_{\ex}^{\tr} H'(\hat \zz)\|=\|\uu_{\ex}^{\tr}[H'(\hat \zz)-H'(\hat \zz_{\ex})]\|\leq \|H'(\hat \zz)-H'(\hat \zz_{\ex})\|=\bigO(\varepsilon),\]
and $H'(\hat{\zz})$ has corank one approximately, so that $\|\uu_n-\uu_{\ex}\|=\bigO(\varepsilon)$.
Moreover, using the same analysis above, we obtain that
\begin{equation*}
|\uu_{\ex}^{\tr}
H''(\hat{\zz}_{\ex})(\hat{\zz}-\hat{\zz}_{\ex})^2|=\bigO(\varepsilon^3).
\end{equation*}
Hence, we have  $ |\uu_{\ex}^{\tr}  H(\hat \zz)|=
\bigO(\varepsilon^3).$
Noticing $\|H(\hat \zz)\|=
\bigO(\varepsilon^2)$, we get
\begin{equation}\label{uu4}
 |\uu_n^{\tr}  H(\hat \zz)|\leq |(\uu_n-\uu_{\ex})^{\tr}  H(\hat \zz)|+|\uu_{\ex}^{\tr}  H(\hat \zz)|=
\bigO(\varepsilon^3).
\end{equation}

 Combining (\ref{Taylor}), (\ref{uu3}) and (\ref{uu4}),  we have
\begin{equation}\label{uu5}
|\uu_n^{\tr}  H'(\hat{\zz}) (\hat{\zz}_{\ex}-\hat{\zz})|
=\bigO(\varepsilon^3), \end{equation} which is
equivalent to
\begin{equation}\label{uu6}
|\sigma'_n \, \vv_n^{\tr}
(\hat{\zz}_{\ex}-\hat{\zz})|=\bigO(\varepsilon^3),
\end{equation} where $\vv_n=[1,0,\ldots,0]^T$ is the right singular vector of
$H'(\hat{\zz})$ corresponding to $\sigma'_n$. Hence,
$|\sigma'_n(\hat{z}_{1,\ex}-\hat{z}_1)|=\bigO(\varepsilon^3)$. Based
on (\ref{newzzz3}), we have
\begin{equation}\label{hd}
\sigma'_n=\bigO(\varepsilon^2).
\end{equation}
Moreover, from (\ref{zzz5}), we have
\begin{equation}\label{H1x2}
\|L_1(H)_{\zz=\hat{\zz}}\|= \left\|\frac{\partial H(\hat \zz)}{
\partial z_1}\right\| = \bigO(\varepsilon^2).
\end{equation}
\end{proof}

It is amazing  to notice that not only the first order Max Noether
condition computed according to Theorem \ref{LiZhi} satisfies
(\ref{H1x2}), but also all other Max Noether conditions up to the
order $\mu-2 \geq 0$ satisfy similar conditions:
\begin{equation}\label{LiH}
\|L_i(H)_{\zz=\hat{\zz}}\|=\bigO(\varepsilon^2), ~{\rm for }~
i=0,\ldots,\mu-2.
\end{equation}

\subsection{An Augmented Polynomial  System}
\label{sec:32}

To prove (\ref{LiH})  inductively, we need to introduce an augmented
polynomial system and prove the following theorem.

\begin{theorem}\label{dflNo}
Let us assume that $H(\zz)$ is a polynomial system which has
$\hat{\zz}_{\ex}$ as an isolated exact singular solution with the
 multiplicity $\mu$, the corank of $H'(\hat{\zz}_{\ex})$ is one. Let
$I$ be the ideal generated by polynomials in $H$ and $\{L_0, L_1,
\ldots, L_{\mu-1}\}$ be a closed basis of
$\triangle_{\hat{\zz}_{\ex}}(I)$, where $L_0=D(0,\ldots,0),
L_1=D(1,0, \ldots, 0)$
 and  $ L_k=
P_k+ a_{k,2} D(0,1,\ldots,0)+\cdots+a_{k,n}D(0,\ldots,1)$
constructed according to  Theorem \ref{LiZhi}.  The augmented
polynomial system
\begin{equation}\label{dflH}
G(\mathbf{z}, \lambda):=\left\{
\begin{array}{l}
H(\zz),\\
H'(\zz)\cdot \lambda,\\
\lambda_1-1,
\end{array}
\right.
\end{equation}
where $\lambda=[\lambda_1, \ldots, \lambda_n]^T$  has
 an isolated singular solution $(\hat{\zz}_{\ex},\hat{\lambda}_{\ex})$
with the multiplicity $\mu-1$, where
$\hat{\lambda}_{\ex}=[1,0,\ldots,0]^T$. If $\mu \geq 3$ then the
Jacobian matrix $G'(\hat{\zz}_{\ex},\hat{\lambda}_{\ex})$ has corank
one and
\begin{equation} \label{goal1}
\tilde{L}_1 =\frac{\partial}{\partial z_1}
+2a_{2,2}\frac{\partial}{\partial
\lambda_2}+\cdots+2a_{2,n}\frac{\partial}{\partial \lambda_n}
\end{equation}
satisfies $ \tilde{L}_1(G)_{(\zz,
\lambda)=(\hat{\zz}_{\ex},\hat{\lambda}_{\ex})}=0.$ Moreover,
starting from $\tilde L_0=D(0,\ldots,0)$ and  $\tilde L_1$, for
$2\leq k \leq \mu-2$, the $k$-th order Max Noether condition  of $G$
at $(\hat{\zz}_{\ex}, \hat{\lambda}_{\ex})$ retaining the closedness
 has the following form:
\begin{equation} \label{TLk}
\tilde{L}_k = \tilde P_k+a_{k,2} \frac{\partial}{\partial z_2}
+\cdots+a_{k,n}\frac{\partial}{\partial z_n}+
(k+1)a_{k+1,2}\frac{\partial}{\partial
\lambda_2}+\cdots+(k+1)a_{k+1,n}\frac{\partial}{\partial \lambda_n}
\end{equation}
  where
\begin{equation}\label{TPk}
\begin{array}{l}
\tilde{P}_k
=P_k+\Psi_{n+2}(Q_{k,n+2})+\Psi_{n+3}(Q_{k,n+3})_{\alpha_{n+2}=0}
+\cdots+\Psi_{2n}(Q_{k,2n})_{\alpha_{n+2}=\cdots=\alpha_{2n-1}=0}  
\end{array}
\end{equation}
and
\begin{equation}\label{PTPk}
\begin{array}{l}
Q_{k,n+j}=\Phi_{n+j}(\tilde{P}_k)= 2
a_{2,j}\tilde{L}_{k-1}+\cdots+ k a_{k,j}\tilde{L}_1,~2\leq j\leq n.
\end{array}
\end{equation}

\end{theorem}

\begin{proof}
The Jacobian matrix of
$G(\mathbf{z}, \lambda)$ at $(\hat{\zz}_{\ex},\hat{\lambda}_{\ex})$ is
\begin{align*}
G'(\hat{\zz}_{\ex},\hat{\lambda}_{\ex})=\begin{bmatrix} H'(\hat{\zz}_{\ex}) & ~~0\\
H''(\hat{\zz}_{\ex})\cdot\hat{\lambda}_{\ex} & ~~H'(\hat{\zz}_{\ex})\\
0 & ~~\hat{\lambda}_{\ex}^T
\end{bmatrix},
\end{align*}
 where $H''(\hat{\zz}_{\ex})\cdot\hat{\lambda}_{\ex}=\left[\frac{\partial^2
 H(\hat{\zz}_{\ex})}{\partial {z_1^2}}, \ldots, \frac{\partial^2
 H(\hat{\zz}_{\ex})}{\partial {z_1} \partial {z_n}} \right]$.  Since the corank of $H'(\hat{\zz}_{\ex})$ is one and $L_1=D(1,0, \ldots,
0) \in \triangle_{\hat{\zz}_{\ex}}^{(1)}(I)$, the first column of
$H'(\hat{\zz}_{\ex})$ is a zero vector and the remaining columns of
$H'(\hat{\zz}_{\ex})$ are linearly independent.  Moreover, since
$\hat{\lambda}_{\ex}^T=[1,0,\ldots,0]$, the last  $2 n -1 $ columns
of $G'(\hat{\zz}_{\ex},\hat{\lambda}_{\ex})$ are linearly
independent and its corank is less than one.

If  $\mu\geq 3$,   the second order Max Noether condition  of $H$ at
$\hat{\zz}_{\ex}$ has the form $L_2=D(2,0, \ldots, 0)+a_{2,2}D(0,1,0,
\ldots, 0)+\cdots+a_{2,n}D(0, \ldots, 0,1)$. From $L_2(H)_{\zz=\hat{\zz}_{\ex}}=0$, we have
\begin{equation*}
\frac{1}{2} \frac{\partial^2 H(\hat{\zz}_{\ex})}{\partial z_1^2}+ a_{2,2}
\frac{\partial
 H(\hat{\zz}_{\ex})}{\partial {z_2}}
\cdots+ a_{2,n}\frac{\partial
 H(\hat{\zz}_{\ex})}{\partial {z_n}}=0.
 \end{equation*}
The vector $\vv=[1,0,\ldots,0,2a_{2,2},\ldots,2a_{2,n}]^T$ is a null
vector of $G'(\hat{\zz}_{\ex},\hat{\lambda}_{\ex})$.  Therefore, the Jacobian
matrix $G'(\hat{\zz}_{\ex},\hat{\lambda}_{\ex})$ has corank one and the first
order differential operator $\tilde L_1$ in (\ref{goal1}) satisfies
\begin{equation}\label{Q1}
\tilde{L}_1(H'(\zz)\cdot \lambda)_{(\zz,
\lambda)=(\hat{\zz}_{\ex},\hat{\lambda}_{\ex})}=2 L_2(H)_{\zz=\hat{\zz}_{\ex}}=0.
\end{equation}
Hence, we have
\begin{equation}\label{MaxL1}
\tilde L_1(G)_{(\zz, \lambda)=(\hat{\zz}_{\ex},\hat{\lambda}_{\ex})}=0.
\end{equation}

Using similar arguments in \cite{LiZhi:2009} for proving Theorem
\ref{LiZhi}, we can show that the differential operators
 $\tilde
L_k$ defined by formulas (\ref{TLk}), (\ref{TPk}) and (\ref{PTPk})
retain the closedness. It should also be noticed  that
 $\tilde{L}_k$ always contains the differential monomial $D(k,0,\ldots,0)$ and
 there are no differential monomials $D(i,0,\ldots,
0)$ for $i< k$ contained in $\tilde{L}_k$.  Otherwise, we can reduce
them by $\tilde{L}_i$. Moreover, $\frac{\partial}{\partial
\lambda_1}$ is not contained in any $\tilde{L}_k$, otherwise,
$\tilde{L}_k(\lambda_1)_{(\zz,\lambda)=(\hat{\zz}_{\ex},\hat{\lambda}_{\ex})}\neq
0$. Hence, due to the closedness,  there are no differential
operators
$D(\alpha_1,\ldots,\alpha_n,\alpha_{n+1},\ldots,\alpha_{2n})$ with
$\alpha_{n+1}>0$ contained in any $\tilde{L}_k$.

Now let us show that the constructed differential operators
$\tilde{L}_k$ satisfy
\begin{equation}\label{MaxLk}
\tilde{L}_k(G)_{(\zz,
\lambda)=(\hat{\zz}_{\ex},\hat{\lambda}_{\ex})}=0, ~{\text {for}}~
1 \leq k \leq \mu-2. \end{equation}
From (\ref{MaxL1}), we can see
that (\ref{MaxLk}) is true for $k=1$. Moreover, it is easy to check
that
\begin{equation}\label{cond1}
\tilde{L}_k(H)_{(\zz,
\lambda)=(\hat{\zz}_{\ex},\hat{\lambda}_{\ex})}=\left(P_k+a_{k,2}\frac{\partial}{\partial
z_2} +\cdots+a_{k,n}\frac{\partial}{\partial
z_n}\right)(H)_{\zz=\hat{\zz}_{\ex}}=0,
\end{equation}
and
\begin{equation}\label{cond2}
\tilde{L}_k(\lambda_1)_{(\zz, \lambda)=(\hat{\zz}_{\ex},\hat{\lambda}_{\ex})}=0.
\end{equation}
Based on formulas  (\ref{TLk}), (\ref{TPk}) and (\ref{PTPk}),  we
have
\begin{align*}\label{PQ}
&\tilde{L}_k (H'(\zz)\cdot \lambda)_{(\zz,
\lambda)=(\hat{\zz}_{\ex},\hat{\lambda}_{\ex})}=\\
&\left(L_k\frac{\partial}{\partial z_1}+\sum_{j=2}^n
(2\,a_{2,j}L_{k-1}+\cdots+k\,a_{k,j}L_{1}+(k+1) \,
a_{k+1,j})\frac{\partial}{\partial
z_j}\right)(H)_{\zz=\hat{\zz}_{\ex}}.
\end{align*}
Let us set
\begin{equation}
Q_{k+1}=L_k\frac{\partial}{\partial z_1}+\sum_{j=2}^n
(2\,a_{2,j}L_{k-1}+\cdots+k\,a_{k,j}L_{1})\frac{\partial}{\partial
z_j}.
\end{equation}
We  show (Proposition~\ref{lemmaQk} in Appendix) that
\begin{equation}\label{Qk}
Q_{k+1}=(k+1)P_{k+1}.
\end{equation}
Hence,  we have
\begin{equation}\label{cond3}
\begin{array}{ll}
&\tilde{L}_k (H'(\zz)\cdot \lambda)_{(\zz,
\lambda)=(\hat{\zz}_{\ex},\hat{\lambda}_{\ex})}\\
&=\left((k+1) P_{k+1}+ (k+1) a_{k+1,2} \frac{\partial}{\partial z_2}
+ \cdots+ (k+1) a_{k+1,n} \frac{\partial}{\partial
z_n}\right)(H)_{\zz=\hat{\zz}_{\ex}}\\
&=(k+1) L_{k+1}(H)_{\zz=\hat{\zz}_{\ex}}=0.  
\end{array}
\end{equation}
From (\ref{cond1}), (\ref{cond2}) and (\ref{cond3}),  we derive that
(\ref{MaxLk}) is true for $ 1 \leq k \leq \mu-2$.
\end{proof}

\begin{corollary}\label{coro}
Suppose $F(\xx)$ is a polynomial system which has $\hat{\xx}_{\ex}$
as an isolated exact singular solution with the multiplicity $\mu$
and the corank of $F'(\hat{\xx}_{\ex})$ is one. Let $\rr_1$ be the
null vector of $F'(\hat{\xx}_{\ex})$ and $\|\rr_1\|=1$. For any
random vector $\mathbf{h}\in \mathbb{C}^n$ satisfying
$\mathbf{h}^{\tr}\rr_1\neq 0$, the augmented polynomial system
\begin{equation}\label{DZconj}
J(\mathbf{x}, \nu):=\left\{
\begin{array}{l}
F(\xx),\\
F'(\xx)\cdot \nu,\\
\mathbf{h}^{\tr}\nu-1,
\end{array}
\right.
\end{equation}
has $(\hat{\xx}_{\ex},\frac{\rr_1}{\mathbf{h}^{\tr}\rr_1})$ as an
isolated singular solution with the multiplicity
 $\mu-1$.
\end{corollary}
\begin{proof}
Let $\{\rr_1,\ldots,\rr_n\}$ be a normal orthogonal basis of
$\mathbb{C}^n$, then
$\mathbf{h}=(\mathbf{h}^{\tr}\rr_1)\rr_1+\cdots+(\mathbf{h}^{\tr}\rr_n)\rr_n$.
If $\mathbf{h}^{\tr}\rr_1\neq 0$,   performing the linear
transformation
\[\xx=R\,\zz,\quad\nu=R\,\lambda,\]
where
$R=\left[\frac{\rr_1}{\mathbf{h}^{\tr}\rr_1},\rr_2-\frac{\mathbf{h}^{\tr}\rr_2}{\mathbf{h}^{\tr}\rr_1}\rr_1,\ldots,\rr_n-
\frac{\mathbf{h}^{\tr}\rr_n}{\mathbf{h}^{\tr}\rr_1}\rr_1 \right]$ is
a regular matrix, we
obtain the augmented polynomial system 
\begin{equation*}
G(\mathbf{z}, \lambda):=\left\{
\begin{array}{l}
H(\zz),\\
H'(\zz)\cdot \lambda,\\
\lambda_1-1,
\end{array}
\right.
\end{equation*}
where
\[H(\zz)=F(R\,\zz),\quad H'(\zz)\cdot \lambda=F'(\xx)\cdot R \cdot R^{-1} \nu,\quad \lambda_1-1=\mathbf{h}^{\tr}\nu-1.\]
According to Theorem~\ref{dflNo}, we know that
$(\hat{\zz}_{\ex},\hat{\lambda}_{\ex})$ is an isolated singular
solution of $G$ with  the multiplicity $\mu-1$,
 where $\hat{\zz}_{\ex}=R^{-1}\hat{\xx}_{\ex}$ and $\hat{\lambda}_{\ex}=[1,0,\ldots,0]^T$.
Hence, by Theorem \ref{lineartrans} and Remark \ref{regular},
$\left(\hat{\xx}_{\ex},\frac{\rr_1}{\mathbf{h}^{\tr}\rr_1}\right)$
is an isolated singular solution of $J(\xx,\nu)$ with the
multiplicity $\mu-1$.
\end{proof}

\begin{remark}
It is well known that the augmented polynomial system $J(\xx,\nu)$
defined in (\ref{DZconj}) has an isolated singular solution
$\left(\hat{\xx}_{\ex},\frac{\rr_1}{\mathbf{h}^{\tr}\rr_1}\right)$
with the multiplicity less than $\mu$, see  \cite{LVZ:2006,
DaytonZeng:2005}.  Here, we proved the  conjecture in
\cite{DaytonZeng:2005} that
  the multiplicity of the singular solution
 of the augmented polynomial system (\ref{DZconj}) drops by one exactly  in the breadth one
 case.
\end{remark}

\begin{remark}\label{remarkdfl}
For the system $H(\zz)$ and its  approximate singular solution $\hat
\zz$ defined in (\ref{zzz1}) and (\ref{zzz22}), the augmented
polynomial system defined in (\ref{dflH}) has
$(\hat{\zz},\hat{\lambda})$ ($\hat \lambda=[1,0, \ldots,0]^T$) as an
approximate solution. Suppose  $\{L_0, \ldots, L_{\mu-1}\}$ is  a
closed basis of the approximate Max Noether space of the system $H$
at $\hat \zz$ constructed according to  Theorem \ref{LiZhi}, from
$L_0=D(0,\ldots,0)$ and $L_1=D(1,0,\ldots,0)$, then
$\{\tilde{L}_0,\tilde{L}_1,\ldots, \tilde{L}_{\mu-2}\}$ constructed
according to Theorem \ref{dflNo} is a closed basis of
the approximate Max Noether space of the system $G$ at
$(\hat{\zz},\hat{\lambda})$, satisfying
\[
\left\{
\begin{array}{l}
\|\tilde{L}_k(H)_{(\zz, \lambda)=(\hat{\zz},\hat{\lambda})}\|=\|L_k(H)_{\zz=\hat{\zz}}\|=\bigO(\varepsilon),\\
\|\tilde{L}_k(H'\lambda)_{(\zz, \lambda)=(\hat{\zz},\hat{\lambda})}\|=\|(k+1) L_{k+1}(H)_{\zz=\hat{\zz}}\|=\bigO(\varepsilon),\\
\|\tilde{L}_k(\lambda_1)_{(\zz,
\lambda)=(\hat{\zz},\hat{\lambda})}\|=0,
\end{array}
\right.
\]
for  $1\leq k\leq \mu-2$.
\end{remark}

\begin{theorem}\label{thm:nextmulti}
Let $F(\xx)$ be a polynomial system which has $\hat{\xx}_{\ex}$ as
an isolated exact singular solution with the multiplicity $\mu$ and
the breadth one. Suppose $\hat{\xx}$ is an approximate solution of
$F$ which satisfies
\begin{equation}\label{inicond}
\|\hat{\xx}-\hat{\xx}_{\ex}\|=\Theta(\varepsilon) \mbox{ and }
\|F(\hat{\xx})\|=\bigO(\varepsilon^2),
\end{equation}
for a small positive number $\varepsilon$.  Let
$\sigma_1,\ldots,\sigma_n$ be the singular values of $F'(\hat{\xx})$
satisfying $\sigma_i=\Theta(1),~1\leq i\leq n-1$ and
$\sigma_n=\bigO(\varepsilon)$. Suppose $\rr_1$ is the right singular
vector corresponding to $\sigma_n$. We form  a unitary matrix
$R=[\rr_1,\ldots,\rr_n]$ and set $H(\zz)=F(R\,\zz)$. Suppose
$\{L_0,\ldots,L_{\mu-1}\}$ is a closed basis of the approximate Max
Noether space of the system $H$ at $\hat{\zz}=R^{-1}\,\hat{\xx}$
constructed according to  Theorem \ref{LiZhi} from
$L_0=D(0,\ldots,0)$ and $L_1=D(1,0,\ldots,0)$, then
\begin{equation*}
\|L_i(H)_{\zz=\hat{\zz}}\|=\bigO(\varepsilon^2), ~{\rm for }~
i=0,\ldots,\mu-2.
\end{equation*}
\end{theorem}

\begin{remark}
Under Assumption \ref{assump}, according to Theorem
\ref{thm:prenewton}, we can always perform the regularized
Newton iteration to obtain an approximate singular solution
$\hat \xx$ satisfying (\ref{inicond}).  Moreover, it should also be
noticed that all discussions in Section \ref{sec:31} after Theorem
\ref{thm:prenewton}  are valid if we start with an approximate
singular solution satisfying (\ref{inicond}).
\end{remark}
\begin{proof}
According to (\ref{zzz4h}) and Theorem \ref{J1}, we know that
Theorem \ref{thm:nextmulti} is true for $\mu=2$ and $\mu=3$.

Now let us assume that Theorem \ref{thm:nextmulti}  is true for
$\mu=k$ and $k\geq3$.  For $\mu=k+1$,   we form the augmented
polynomial system
 $G(\zz,\lambda)=\{H(\zz), H'(\zz) \cdot \lambda, \lambda_1-1\}$.

  According to
Theorem \ref{lineartrans}, the root $\hat{\zz}_{\ex}$ defined in
(\ref{zzz2}) is an exact  singular solution of $H(\zz)$ with the
multiplicity $\mu$ and the corank of $H'(\hat{\zz}_{\ex})$ is one.
Let $\vv$ be the null vector of $H'(\hat{\zz}_{\ex})$ and
$\|\vv\|=1$. Since
\begin{equation}\label{hv}
\|H'(\hat{\zz})\vv\|=\|[H'(\hat{\zz})-H'(\hat{\zz}_{\ex})]\vv\|=
\bigO(\varepsilon),
\end{equation}
$\left[\frac{\partial H(\hat{\zz})}{\partial
z_2},\ldots,\frac{\partial H(\hat{\zz})}{\partial z_n}\right]$ is of
full column rank,  combining with  (\ref{zzz5}), we derive that
\begin{equation}\label{v1i}
v_1=\Theta(1), ~{\text {and}}~  v_i=\bigO(\varepsilon),
~{\text{for}}~ 2 \leq i \leq n.
\end{equation}
Set $\mathbf{h}=[1,0,\ldots, 0]^T$,  we have $\mathbf{h}^{\tr}\vv
=v_1=\Theta(1) \neq 0$.

According to Corollary \ref{coro},  the augmented polynomial system
$G(\zz, \lambda)$ has
 $\left(\hat{\zz}_{\ex},\frac{\vv}{\mathbf{h}^{\tr}
 \vv}\right)=(\hat{\zz}_{\ex},\hat{\lambda}_{\ex})$,
 where
  \[\hat{\lambda}_{\ex}= \left[ 1, \frac{v_2}{v_1}, \ldots, \frac{v_n}{v_1}
  \right]^T,\]
 as an
isolated singular solution with the multiplicity $\mu-1$, which is
equal to $k$. According to Remark \ref{remarkdfl},
$(\hat{\zz},\hat{\lambda})$ ($\hat \lambda=[1,0, \ldots,0]^T$) is an
approximate solution of $G(\zz, \lambda)$. Moreover, by (\ref{zzz4})
and (\ref{v1i}), we have
\[\|(\hat{\zz},\hat{\lambda})-(\hat{\zz}_{\ex},\hat{\lambda}_{\ex})\|=\sqrt{\|\hat{\zz}-\hat{\zz}_{\ex}\|^2+\|\hat{\lambda}-\hat{\lambda}_{\ex}\|^2}=\Theta(\varepsilon).\]
Furthermore,  from (\ref{zzz4h}) and (\ref{H1x2}), we have
\begin{equation*}
\|G(\hat{\zz},\hat{\lambda})\| =\sqrt{\|H(\hat{\zz})\|^2+
\left\|\frac{\partial H(\hat \zz)}{
\partial z_1}\right\|^2}=\bigO(\varepsilon^2).
\end{equation*}

We have assumed that  Theorem \ref{thm:nextmulti}  is true when the
multiplicity  is equal to  $k$. Therefore,  for the augmented
polynomial system $G(\zz, \lambda)$, we can form a unitary matrix
$\bar{R}$ with
$\rr_1=\frac{1}{a}[1,0,\ldots,0,2a_{2,2},\ldots,2a_{2,n}]^T$ as its
first column, where $a=\sqrt{1+ 4 (a_{2,2}^2+\cdots+a_{2,n}^2)}$,
then generating a new system $J(\ww)=G(\bar{R}\,\ww)$ which has an
approximate singular solution $\hat{\ww}$ with the multiplicity $k$.
By the inductive assumption, we have
\[\|\bar{L}_i(J)_{\ww=\hat{\ww}}\|=\bigO(\varepsilon^2),  ~{\text{for}}~  0\leq
i\leq k-2,\]
 where $\bar{L}_i$ is the $i$-th Max Noether condition
of $J$ at $\hat{\ww}$ constructed by Theorem \ref{LiZhi} from
$\bar{L}_0=D(0,\ldots,0)$ and $\bar{L}_1=D(1,0,\ldots,0)$. According
to Theorem \ref{lineartrans},
\[\bar{L}_i(J)_{\ww=\hat{\ww}}=\Gamma_{\bar{R}}(\bar{L}_i)(G)_{(\zz,\lambda)=(\hat{\zz},\hat{\lambda})}.\]
Since $\{\tilde{L}_0,\tilde{L}_1,\ldots, \tilde{L}_{k-1}\}$ and
$\{\Gamma_{\bar{R}}(\bar{L}_0),\Gamma_{\bar{R}}(\bar{L}_1),\ldots,
\Gamma_{\bar{R}}(\bar{L}_{k-1})\}$ are both closed basis
of the approximate Max Noether space of the system $G$ at
$(\hat{\zz},\hat{\lambda})$, and
\[ \Gamma_{\bar{R}}(\bar{L}_0)=\tilde{L}_0, ~{\text{and}}~
\Gamma_{\bar{R}}(\bar{L}_1)=\frac{1}{a}\tilde{L}_1,\]
 we derive that
$\Gamma_{\bar{R}}(\bar{L}_i)$ is a linear combination of
$\{\tilde{L}_0,\tilde{L}_1,\ldots, \tilde{L}_{i}\}$ (Proposition
\ref{prop3} in Appendix). Hence, we have
$\|\tilde{L}_i(G)_{(\zz,\lambda)=(\hat{\zz},\hat{\lambda})}\|=\bigO(\varepsilon^2)$,
and
\[
\|L_{i+1}(H)_{\zz=\hat{\zz}}\|=\left\|\frac{1}{i+1}\tilde{L}_{i}(H'\lambda)_{(\zz,\lambda)
=(\hat{\zz},\hat{\lambda})}\right\|=\bigO(\varepsilon^2).
\]
Therefore, Theorem \ref{thm:nextmulti}  is true for $\mu=k+1$.
\end{proof}

\subsection{An Algorithm  for Refining Approximate Singular Solutions}
\label{sec:34}

\noindent
\begin{algorithm}\label{MultiStructure}{\sf
MultipleRootRefinerBreadthOne}

\noindent
 \textbf{Input:} An approximate solution $\hat{\xx}$ of a polynomial system
$F$ which is close to an isolated exact singular solution of $F$
with the multiplicity $\mu$ in the breadth one case, and a tolerance
$\tau$.

\noindent
 \textbf{Output:} Refined solution $\hat{\xx}$.

\begin{enumerate}

\item  \label{step:1}{\sf Regularized Newton Iteration:} Solve the regularized least squares problem
\begin{equation*}
({A}^{\tr} {A}+\sigma_n I_n)\hat{\yy} ={A}^{\tr} \bb,
\end{equation*}
 where $\bb=-F(\hat{\xx})$, $A^{\tr}$ is the Hermitian (conjugate) transpose of $A=F'(\hat{\xx})$, $I_n$ is the $n\times n$ identity matrix and $\sigma_n$ is
 the smallest singular value of $A$.

\item \label{step:2}Compute the  null vector $\rr_1$ of  $F'(\hat{\xx}+\hat{\yy})$ with respect to $\tau$,  form a unitary matrix $R$ with $\rr_1$ as its first column
and perform the linear transformation
\[H(\zz):=F(R\,\zz),\]
and set $\hat{\zz}:=R^{-1}(\hat{\xx}+\hat{\yy})$.

\item \label{step:3}
Construct a closed basis of the approximate Max Noether space of
$I=(h_1,\ldots,h_n)$ at $\hat{\zz}$ with respect to $\tau$:
\[\triangle_{\hat{\zz}}^{(\mu-1)}({I}):=\Span (L_0,
 L_1,\ldots,L_{\mu-1})\]
by Algorithm {\sf MultiplicityStructureBreadthOneNumeric} in
\cite{LiZhi:2009}.

\item \label{step4} Solve the linear system
\begin{equation}\label{Pk}
\left[P_{\mu}(H)_{\zz=\hat{\zz}},\frac{\partial H(\hat
\zz)}{\partial z_2}, \ldots, \frac{\partial H(\hat \zz)}{\partial
z_n}\right]\,\vv=-L_{\mu-1}(H)_{\zz=\hat{\zz}},
\end{equation}
 where
$\vv=[v_1,\cdots,v_{n}]^T$ and $P_{\mu}$ is the differential
operator of order $\mu$ computed by formulas in Theorem \ref{LiZhi}. Set
$\delta:=\frac{v_1}{\mu}$.

\item \label{step:5} Return
\begin{equation*}
   \hat{\xx}:= \hat{\xx}+\hat{\yy}+\delta\,\rr_1.
\end{equation*}
\end{enumerate}
\end{algorithm}

\begin{remark}
The size of matrices involved in the algorithm {\sf
MultipleRootRefinerBreadthOne} is bounded by $n\times n$, whereas
the size of matrices used  in the deflation method
 is bounded by $ (\mu \,
n ) \times (\mu \, n)$ \cite{DaytonZeng:2005, LVZ:2006}.
\end{remark}


\begin{remark}
In fact, in order to keep the sparse structure of the original
polynomial system, we should avoid performing the linear
transformation. Moreover, it is  expensive to compute and store all
Max Noether conditions. Since  we only need their evaluations  to
solve
 (\ref{Pk}), it's possible to compute and store only the necessary
 evaluations of these Max Noether conditions.  We will discuss these
 issues in forthcoming papers.
\end{remark}

\subsection{Quadratic Convergence of the Algorithm}
\label{sec:33}

\noindent
\begin{theorem}\label{thmmu23}
Under Assumptions \ref{assump}, the
refined singular solution $\hat \xx$ returned by Algorithm {\sf
MultipleRootRefinerBreadthOne} satisfies
\begin{equation}\label{convm}
\|\hat{\xx}-\hat{\xx}_{\ex}\|=\bigO(\varepsilon^2).
\end{equation}
\end{theorem}

\begin{proof}
According to  Theorem \ref{thm:nextmulti}, we have
$L_{i}(H)_{\zz=\hat{\zz}}=\bigO(\varepsilon^2)$, for $0\leq i \leq
\mu-2$. Since
\begin{equation*}
 \Phi_k({L}_i) \in \Span (L_0,
\ldots, L_{\mu-2}),  ~{\text{for}}~ 1 \leq k \leq n,
\end{equation*}
we have
\begin{equation}\label{meq1}
\|{L}_i((z_k-\hat{z}_k)
H)_{\zz=\hat{\zz}}\|=\|\Phi_k({L}_i)(H)_{\zz=\hat{\zz}}\|=\bigO(\varepsilon^2),
  ~{\text{for}}~ 1 \leq k \leq n,  0 \leq i \leq \mu-1.
\end{equation}

The matrix in (\ref{Pk}) is of full rank. We solve
the linear system (\ref{Pk}) to obtain the vector $\vv=[v_1, \ldots,
v_n]^T$ such that $L_{\mu}(H)_{\zz=\hat{\zz}}=0$ for
\begin{equation}\label{updateLu}
L_{\mu}:=L_{\mu-1}+v_1\cdot P_{\mu}+v_2\cdot \frac{\partial
}{\partial z_2}+\cdots+v_{n}\cdot \frac{\partial }{\partial z_n}.
\end{equation}
It should be noticed that the vector $\vv$ satisfies
$\|\vv\|=\bigO(\varepsilon)$ since $\|L_{\mu-1}(H)_{\zz=\hat
\zz}\|=\bigO(\varepsilon)$. Moreover,
\begin{equation*}
 \Phi_k({L}_{\mu}) \in \Span (L_0,
\ldots, L_{\mu-2}),  ~{\text{for}}~ 2 \leq k \leq n,
\end{equation*}
we have
\begin{equation}
\|{L}_{\mu}((z_k-\hat{z}_k)
H)_{\zz=\hat{\zz}}\|=\|\Phi_k({L}_{\mu})(H)_{\zz=\hat{\zz}}\|=\bigO(\varepsilon^2),
  ~{\text{for}}~ 2 \leq k \leq n.
\end{equation}
For $k=1$, since
 $ \|\Phi_1(v_1 {P}_{\mu})(H)_{\zz=\hat \zz}\|=\|v_1 L_{\mu-1}(H)_{\zz=\hat \zz}\|=\bigO(\varepsilon^2)$,
we have
\begin{equation}\label{meq2}
\|{L}_{\mu}((z_1-\hat{z}_1)
H)_{\zz=\hat{\zz}}\|=\|\Phi_1({L}_{\mu})(H)_{\zz=\hat{\zz}}\|=\bigO(\varepsilon^2).
\end{equation}
From (\ref{meq1}) and (\ref{meq2}), for $i=0, 1,\ldots, \mu-2, \mu$,
we have
\begin{equation*}\label{etli}
 \|{L}_i
(p\cdot H)_{\zz=\hat{\zz}}\|=\bigO(\varepsilon^2), ~ \forall p\in
\{(z_1-\hat{z}_1)^{\alpha_1}\cdots (z_n-\hat{z}_n)^{\alpha_n},
\alpha_1 \geq 0, \ldots, \alpha_n \geq 0\}.
\end{equation*}
Especially, we have
\[\|M_{\mu+1}\cdot L_i (\vv(\zz)_{\mu})_{\zz=\hat{\zz}}\|=\bigO(\varepsilon^2),\]
where $M_{\mu+1}$ is the coefficient matrix of the Taylor expansion
of  the system $H$ and  all its prolongations up to the  degree
$\mu$ at $\hat \zz$, and
\[\vv(\zz)_{\mu}=\left[(z_1-\hat{z}_1)^{\mu},(z_1-\hat{z}_1)^{\mu-1}(z_2-\hat{z}_2),\ldots,z_1-\hat{z}_1,\ldots,z_n-\hat{z}_n,1\right]^T.\]

It is important to notice that, based on the closedness conditions,
we obtain the null space of $M_{\mu+1}$ with  matrices of size $n
\times n$ instead of generating the big matrix $M_{\mu+1}$.
 Similarly to the
analysis in \cite[Remark 18]{WuZhi:2009}, the trace of the
multiplication matrix $\widetilde{M}_{z_1}$ formed from approximate
null vectors $L_i(\vv(\zz)_{\mu})_{\zz=\hat{\zz}}$ has the following
property
\begin{equation}\label{doublecorrect}
\frac{1}{\mu} \Tr(\widetilde{M}_{z_1})=\frac{1}{\mu}
\Tr(M_{z_1})+\bigO(\varepsilon^2)=-\hat{z}_{1,\er}+\bigO(\varepsilon^2).
\end{equation}
It is interesting to notice that,  by using the approximate basis
$\{L_0,\ldots,L_{\mu-2},L_{\mu}\}$  and  the normal set
$\left\{1,\frac{\partial }{\partial
z_1},\ldots,\frac{\partial^{\mu-1} }{\partial z^{\mu-1}_1}\right\}$,
we can form the multiplication matrix
\begin{equation*}
\widetilde{M}_{z_1}\cdot \left[\begin{array}{cccc}
l_0~ &~ 0 ~ &~  \cdots ~& ~ 0\\
0~  &~ l_1 ~ &~ \ddots ~& ~ \vdots \\
\vdots~  &~  \ddots ~ &~  \ddots ~& ~ 0\\
0 ~&~ \cdots ~ &~ 0 ~ &~  l_{\mu-1}\\
 \end{array} \right]=\left[\begin{array}{ccccc}
0~ &~ l_1 ~ &~ 0 ~ &~  \cdots ~& ~ 0\\
0~  &~ 0 ~ &~ l_2 ~&~ \ddots ~& ~ \vdots \\
\vdots~  &~  \ddots ~&~  \ddots ~ &~  \ddots ~& ~ 0\\
\vdots ~&~  ~&~ \ddots ~ &~ \ddots ~ &~  l_{\mu-1}\\
0 ~&~ \cdots ~&~ \cdots ~ &~ 0 ~ &~  v_1\cdot l_{\mu-1}\\
 \end{array} \right],
\end{equation*}
where $l_i$ is the coefficient of $\frac{\partial^{i} }{\partial
z^{i}_1}$ in $L_{i}$. Hence, the trace of $\widetilde{M}_{z_1}$ is
$v_1$. Therefore,  there is no need to form the multiplication
matrix! According to (\ref{doublecorrect}), we have
\begin{equation}
\frac{v_1}{\mu}+\hat{z}_{1,\er}=\bigO(\varepsilon^2).
\end{equation}

 Since the last $n-1$ elements of $\hat{\zz}$ have
already been refined quadratically, by updating
$\hat{z}_1:=\hat{z}_1+\delta$ for $\delta:=\frac{v_1}{\mu}$,  we
have
\[ \|\hat{\xx}-\hat{\xx}_{\ex}\|= \| R \, ( \hat \zz- \hat
\zz_{\ex}) \|=\bigO(\varepsilon^2).
\]
\end{proof}

\begin{remark}
The algorithm {\sf MultipleRootRefinerBreadthOne} also works well
for some overdetermined polynomial systems, i.e., the number of
polynomials is bigger than the number of variables, see the last
example Menzel1 in  Table \ref{table1}.
\end{remark}

\section{Examples}\label{lsec:complexity}

The following experiments are done in Maple~13 under Linux for
$\text{Digits}:=15$. Let  $t$ and  $s$ be the number of polynomials
and variables respectively, $\mu$ be the multiplicity. Systems  DZ3,
Dayton2 and DLZ are quoted from \cite{Dayton:2007,
DLZ:2009,DaytonZeng:2005}, Menzel1 and SY5 are cited from
\cite{Menzel84} and \cite{ShenYpma05} respectively. Other examples
are cited from the PHCpack demos by Jan Verschelde.

\begin{table}[ht]
\begin{center}

\begin{tabular}
{|@{\hbox to +0.4em{\hss}} r@{\hbox to +0.4em{\hss}} |@{\hbox to
+0.4em{\hss}}c@{\hbox to +0.4em{\hss}} |@{\hbox to
+0.4em{\hss}}c@{\hbox to +0.4em{\hss}} |@{\hbox to
+0.4em{\hss}}c@{\hbox to +0.4em{\hss}} |@{\hbox to
+0.4em{\hss}}c@{\hbox to +0.4em{\hss}} |@{\hbox to
+0.4em{\hss}}l@{\hbox to +0.4em{\hss}} |} \hline
 System &Zero
&$t$ &$s$ &$\mu$   &\# Digits
\\
\hline
  Ojika1 & $(1,2)$ & $2$ & $2$ & $3$ &  $2 \rightarrow 5 \rightarrow 11 \rightarrow 15$\\
\hline
  Ojika2 & $(1,0,0)$ & $3$ & $3$ & $2$ & $2  \rightarrow 5 \rightarrow 10 \rightarrow 14$ \\
\hline
  Ojika3 & $(-2.5,2.5,1)$ & $3$ & $3$ & $2$  & $2 \rightarrow 4 \rightarrow 9 \rightarrow 14$ \\
\hline
  Ojika4 & $(0,0,10)$ & $3$ & $3$ & $3$  & $2 \rightarrow 3 \rightarrow 7 \rightarrow 13$ \\
\hline
  Decker2 & $(0,0)$ & $2$ & $2$ & $4$  & $2  \rightarrow 5 \rightarrow 15$\\
\hline
  DLZ & $ (0,0)$ & $2$ & $2$ & $10$      & $2 \rightarrow 5 \rightarrow 16$\\
\hline
  DZ3 & $(\frac{2\sqrt{7}}{5}+\frac{\sqrt{5}}{5},-\frac{\sqrt{7}}{5}+\frac{2\sqrt{5}}{5})$ & $2$ & $2$ & $5$     & $2 \rightarrow 5 \rightarrow 13$\\
\hline
  Dayton2 & $(0,0,0)$ & $3$ & $3$ & $5$ & $ 2  \rightarrow 3 \rightarrow 7 \rightarrow 13$ \\
\hline
  SY5 & $(1,1)$ & $2$ & $2$ & $2$ & $ 2  \rightarrow 5 \rightarrow 11 \rightarrow 14$ \\
\hline
  Menzel1 & $(1,1)$ & $3$ & $2$ & $2$ & $ 2  \rightarrow 5 \rightarrow 10 \rightarrow 14$ \\
\hline
\end{tabular}
\vspace{0.2cm}

 {\label{table1}Algorithm Performance}
\end{center}

\end{table}

\section{Conclusion}\label{lsec:conclusion}

It is  a challenge problem to solve the polynomial systems with
singular solutions. Various symbolic-numeric methods have been
proposed for refining  an approximate singular solution to high
accuracy~\cite{Corless:1997,DLZ:2009,DaytonZeng:2005,  GLSY:2005,
GLSY:2007,  Lecerf:2002, LVZ:2006,
 Ojika:1987, WuZhi:2008, WuZhi:2009}. The breadth one case root refinement  has been studied
in~\cite{DLZ:2009,DaytonZeng:2005,GLSY:2007,Griewank85}. In this
paper, we show how to apply strategies in \cite{LiZhi:2009} to
reduce the size of matrices appeared in
\cite{DaytonZeng:2005,WuZhi:2008} to obtain a more efficient
algorithm for refining an approximately known multiple root for this
special  case. We have proved the quadratic convergency of the new
algorithm  when the approximate solution is close to the isolated
exact singular solution. We also notice that when the singular
solution $\hat{\xx}_{\ex}$ is not well separated from other
solutions of $F$, it is  difficult to ensure that the approximate
solution $\hat \xx$ will   converge to $\hat{\xx}_{\ex}$. In
\cite{RumpGraillat:2009}, they described an algorithm for computing
verified error bounds for double  roots of polynomial systems.  We
will  explore  ways of computing the certified  bound for
$\varepsilon$ to guarantee the convergency of our algorithm. It is
also interesting to see whether the approach in the paper can be
generalized to refine singular solutions when the Jacobian matrix is
not of corank one.

\bibliographystyle{siam}
\bibliography{EKBib/strings,EKBib/kaltofen,EKBib/new,zhi,wuxiaoli,EKBib/crossrefs}

{\small
\paragraph{\bf Appendix}
%

In the following, we suppose that there are no differential
operators $D(i,0,\ldots, 0)$ for $i< k$ contained in $\tilde{L}_k$,
otherwise, we can reduce them by $\tilde{L}_i$. Here and hereafter,
we always assume the coefficient of $D(k,0,\ldots,0)$ is one.

\begin{proposition}\label{lemmaQk}
The formula
 (\ref{Qk}) is true  for $ 1 \leq k \leq \mu-1$.
\end{proposition}
\begin{proof}
From (\ref{Q1}), we know  that (\ref{Qk}) is true for $k=1$. Now let
us assume (\ref{Qk}) is true for $k$.
 Let $C_{D(\alpha)}^{P_{k+1}}$ and $C_{D(\alpha)}^{Q_{k+1}}$
denote the coefficients of $D(\alpha)$ in $P_{k+1}$ and $Q_{k+1}$
respectively.
 In order to prove (\ref{Qk}) for $k+1$,
we   show that
$C_{D(\alpha)}^{Q_{k+1}}=(k+1)C_{D(\alpha)}^{P_{k+1}}$ by using the
following relations repeatedly:
\begin{equation}\label{coeff}
\left\{
\begin{array}{ll}
C_{D(\alpha)}^{P_{k+1}}=C_{\Phi_{1}(D(\alpha))}^{L_{k}}, & ~\alpha_1\neq 0,\\
C_{D(\alpha)}^{P_{k+1}}=a_{2,j}C_{\Phi_{j}(D(\alpha))}^{L_{k-1}}+\cdots+a_{k,j}C_{\Phi_{j}(D(\alpha))}^{L_{1}}, & ~\alpha_j\neq 0.\\
\end{array}
\right.
\end{equation}

 Let $D(\alpha)=D(\alpha_1, \ldots, \alpha_n)$,
 denote $j_1=\cdots=j_{\alpha_1}=1,\ldots,j_{|\alpha|-\alpha_i+1}=\cdots=j_{|\alpha|}=i$, where  $\alpha_i$
 is the last nonzero entry in $\alpha$, e.g.,
$j_1=j_2=j_3=1,j_4=j_5=2,j_6=3$ for $D(\alpha)=D(3,2,1,0)$. Since
all derivatives in $Q_{k+1}$ and $P_{k+1}$ are of order at least
$2$, we can start with $|\alpha| = 2$.

\begin{enumerate}
\item   If $\alpha_1=0$ and $|\alpha|=2$, then  we have
\begin{align*}
C_{D(\alpha)}^{Q_{k+1}}=&2a_{2,j_1}a_{k-1,j_2}+\cdots+(k-1)a_{k-1,j_1}a_{2,j_2}\\
&+2a_{2,j_2}a_{k-1,j_1}+\cdots+(k-1)a_{k-1,j_2}a_{2,j_1}\\
=&(k+1)(a_{2,j_1}a_{k-1,j_2}+\cdots+a_{k-1,j_1}a_{2,j_2})\\
=&(k+1)C_{D(\alpha)}^{P_{k+1}}.
\end{align*}
\item   If $\alpha_1\neq0$ and $|\alpha|=2$,
\[C_{D(\alpha)}^{Q_{k+1}}=a_{k,j_1}+k\,a_{k,j_1}=(k+1)a_{k,j_1}=(k+1)C_{D(\alpha)}^{P_{k+1}}.\]
\item   If $\alpha_1\neq0$ and $|\alpha|>2$, by induction,
\begin{align*}
&C_{\Phi_{z_1}(D(\alpha))}^{Q_{k}}=kC_{\Phi_{1}(D(\alpha))}^{P_{k}}\\
&=(\alpha_1-1)C_{\Phi_{1}^2(D(\alpha))}^{L_{k-1}}+ \sum_{j=2}^n
\alpha_j\left(2a_{2,j}C_{\Phi_{1}\Phi_{j}(D(\alpha))}^{L_{k-2}}+\cdots+(k-1)a_{k-1,j}C_{\Phi_{1}\Phi_{j}(D(\alpha))}^{L_{1}}\right).
\end{align*}
While based on (\ref{coeff}), we have
$C_{D(\alpha)}^{P_{k+1}}=C_{\Phi_{1}(D(\alpha))}^{L_{k}}$ and
\begin{align*}
&C_{D(\alpha)}^{Q_{k+1}}=\alpha_1C_{\Phi_{1}(D(\alpha))}^{P_{k}}+\sum_{j=2}^n
\alpha_j\left(2a_{2,j}C_{\Phi_{j}(D(\alpha))}^{P_{k-1}}+\cdots+k  a_{k,j}C_{\Phi_{j}(D(\alpha))}^{P_{1}}\right)\\
&=\alpha_1C_{\Phi_{1}^2(D(\alpha))}^{L_{k-1}}+\sum_{j=2}^n
\alpha_j\left(2a_{2,j}C_{\Phi_{1}\Phi_{j}(D(\alpha))}^{L_{k-2}}+\cdots+(k-1)a_{k-1,j}C_{\Phi_{1}\Phi_{j}(D(\alpha))}^{L_{1}}\right)\\
&=(k+1)C_{\Phi_{1}(D(\alpha))}^{P_{k}}=(k+1)C_{D(\alpha)}^{P_{k+1}}.
\end{align*}
\item   If $\alpha_1=0$ and $|\alpha|>2$, we have
\begin{align*}
C_{D(\alpha)}^{Q_{k+1}}=&2a_{2,j_1}C_{\Phi_{j_1}(D(\alpha))}^{P_{k-1}}+\cdots+k a_{k,j_1}C_{\Phi_{j_1}(D(\alpha))}^{P_{1}}\\
&+\cdots+2a_{2,j_{|\alpha|}}C_{\Phi_{j_{|\alpha|}}(D(\alpha))}^{P_{k-1}}+\cdots+k\cdot
a_{k,j_{|\alpha|}}C_{\Phi_{j_{|\alpha|}}(D(\alpha))}^{P_{1}}.
\end{align*}
By (\ref{coeff}), we have
$C_{\Phi_{j}(D(\alpha))}^{P_{k}}=a_{2,j}C_{\Phi_{j}(D(\alpha))}^{L_{k-1}}+\cdots+a_{k-1,j}C_{\Phi_{j}(D(\alpha))}^{L_{1}}.$

For $2\leq p\leq k$, we  collect all items with respect to
$a_{p,j_1}$:
\begin{align*}
&p
a_{p,j_1}C_{\Phi_{j_1}(D(\alpha))}^{L_{k-p+1}}+a_{p,j_1}\left(2a_{2,j_2}C_{\Phi_{j_1}\Phi_{j_2}(D(\alpha))}^{L_{k-p-1}}
+\cdots+(k-p) a_{k-p,j_2}C_{\Phi_{j_1}\Phi_{j_2}(D(\alpha))}^{L_{1}}\right)\\
&+\cdots+a_{p,j_1}\left(2a_{2,j_{|\alpha|}}C_{\Phi_{j_1}\Phi_{j_{|\alpha|}}(D(\alpha))}^{L_{k-p-1}}+\cdots+(k-p)\cdot
a_{k-p,j_{|\alpha|}}C_{\Phi_{j_1}\Phi_{j_{|\alpha|}}(D(\alpha))}^{L_{1}}\right)\\
&=p\cdot
a_{p,j_1}C_{\Phi_{j_1}(D(\alpha))}^{L_{k-p+1}}+(k-p+1)a_{p,j_1}C_{\Phi_{j_1}(D(\alpha))}^{L_{k-p+1}}=(k+1)a_{p,j_1}C_{\Phi_{j_1}(D(\alpha))}^{L_{k-p+1}}.
\end{align*}
Hence, $C_{D(\alpha)}^{Q_{k+1}}=(k+1)\sum_{p=2}^k
a_{p,j_1}C_{\Phi_{j_1}(D(\alpha))}^{L_{k-p+1}}=(k+1)C_{D(\alpha)}^{P_{k+1}}$.
\end{enumerate}
Hence, (\ref{Qk}) is true for $ 1 \leq k \leq \mu-1$.
\end{proof}

\begin{proposition}\label{prop3}
 In the proof of Theorem \ref{thm:nextmulti}, we claim that
$\Gamma_{\bar{R}}(\bar{L}_i)$  is a linear combination of
$\{\tilde{L}_0,\tilde{L}_1,\ldots,\tilde{L}_i\}$.
\end{proposition}

\begin{proof}
For $i=2$, since
$\Gamma_{\bar{R}}(\bar{L}_1)=\frac{1}{a}\tilde{L}_1$, we can reduce
$\Gamma_{\bar{R}}(\bar{L}_2)$ by $\tilde{L}_1$ to a differential
operator which does not contain  $D(1,0,\ldots,0)$,  denoted by
$\bar{\Gamma}_{\bar{R}}(\bar{L}_2)$. Since $\bar{L}_2$ is
constructed by Theorem \ref{LiZhi} from $\bar{L}_0=D(0,\ldots,0)$
and $\bar{L}_1=D(1,0,\ldots,0)$, $D(2,0\ldots,0)$ is the only second
order  derivative contained in $\bar{L}_2$ with coefficient one
\cite[Lemma 3.3]{LiZhi:2009}. Moreover, by Theorem
\ref{lineartrans}, the coefficient of $D(2,0\ldots,0)$ in
$\Gamma_{\bar{R}}(\bar{L}_2)$ is $\frac{1}{a^2}$. Since
$\bar{\Gamma}_{\bar{R}}(\bar{L}_2)$ is an approximate basis, due to
the closedness,  we have
\begin{equation*}
\Phi_{1}(\bar{\Gamma}_{\bar{R}}(\bar{L}_2))=\frac{1}{a^2}\tilde{L}_{1}.
\end{equation*}
Therefore, we have
\[\bar{\Gamma}_{\bar{R}}(\bar{L}_2)=\frac{1}{a^2}\tilde{L}_{2}.\]
Hence, $\Gamma_{\bar{R}}(\bar{L}_2)$ is a linear combination of
$\{\tilde{L}_0, \tilde{L}_1,\tilde{L}_2\}$. Let us assume that for
all $0\leq i\leq k-1$, the proposition is true.  We can reduce
$\Gamma_{\bar{R}}(\bar{L}_{k})$ by $\{\tilde{L}_0,
\tilde{L}_1,\ldots,\tilde{L}_{k-1}\}$ to a differential operator
which does not contain differential operators $D(i,0,\ldots,0)$,
denoted by $\bar{\Gamma}_{\bar{R}}(\bar{L}_{k})$. Since
$\bar{\Gamma}_{\bar{R}}(\bar{L}_k)$ is an approximate basis, due to
the closedness,  we have
\begin{equation*}
\Phi_{1}(\bar{\Gamma}_{\bar{R}}(\bar{L}_k))=\frac{1}{a^k}\tilde{L}_{k-1},
\end{equation*}
therefore,
$\bar{\Gamma}_{\bar{R}}(\bar{L}_k)=\frac{1}{a^k}\tilde{L}_{k}$.
Hence, $\Gamma_{\bar{R}}(\bar{L}_{k})$ is a linear combination of
$\{\tilde{L}_0,\tilde{L}_1,\ldots,\tilde{L}_{k}\}$.
\end{proof}

}
\end{document}